\documentclass[]{interact}

\usepackage{epstopdf}
\usepackage[caption=false]{subfig}

\theoremstyle{plain}
\newtheorem{theorem}{Theorem}[section]

\newtheorem{corollary}[theorem]{Corollary}
\newtheorem{proposition}[theorem]{Proposition}
\usepackage{tikz-cd}
\theoremstyle{definition}
\newtheorem{definition}[theorem]{Definition}
\newtheorem{example}[theorem]{Example}

\theoremstyle{remark}

\begin{document}
	
	\articletype{ARTICLE TEMPLATE}
	
	\title{Impartial games on two finite Groups}
	
	\author{\name{Ratan Lal\thanks{CONTACT Ratan Lal. Email: vermarattan789@gmail.com}, Muskan\thanks{CONTACT Muskan. Email: muskanmaheshwari0220@gmail.com} and Vipul Kakkar\thanks{CONTACT Vipul Kakkar. Email: vplkakkar@gmail.com}}
		\affil{Desh Bhagat Pandit Chetan Dev Government College of Education, Faridkot, Punjab, India}	
		\affil{Department of Mathematics, Central University of Rajasthan, Rajasthan, India.}
		\affil{Department of Mathematics, Central University of Rajasthan, Rajasthan, India.}
	}
	

	\maketitle
	
	\begin{abstract}
	In this paper, we study impartial achievement games and impartial avoidance games introduced by Anderson and Harary. Using the criteria of maximal subgroups, we study the game for Frobenius groups and non-abelian groups with all abelian subgroups.
	\end{abstract}
	
\begin{amscode}91A46; 20D30\end{amscode}
	
	\section{Introduction}\label{sec1}
	
An \textit{impartial game} is a two player game in which each player chooses its move in a pre-defined manner and both the players know about all possible moves of each other. Anderson and Harary \cite{2} first introduced two impartial games, namely \textit{the achievement game} and \textit{the avoidance game} on finite groups. In this game, each player selects a yet-unselected element of a given finite group until the group is generated. In the achievement game, the player who generates the group from the jointly selected elements wins the game. This game is denoted by $GEN$. In the avoidance game, the player who cannot choose an element from the group without getting a generating set, loses the game. This game is denoted by $DNG$. Such games are studied by many authors as \cite{bn}, \cite{1} and \cite{2}.

Let $S$ be a set of ordinals. Then the \textit{minimum excludant} of the set $A$ is defined as the smallest ordinal not contained in the set $S$. The \textit{nim-value} of a position $P$ is defined as the minimum excludant of the set of nim-values of the options of $P$.
The main problem in the theory of impartial combinatorial games is to find the nim-value of the game. The nim-values are vital as they determine the outcome of the game. Ernst and Sieben \cite{1} developed some theoretical tools using maximal subgroups of a group that allow the determination of the nim-values of the achievement and the avoidance games for a variety of familiar groups. They introduced the \textit{structure diagram} of a game, which is an identification digraph of the game digraph that is compatible with the nim-values of the positions. A \textit{game digraph} is defined as the digraph whose vertices are the positions of the game, every position is connected to its options by arrows and every position is labeled by the nim-value of the corresponding position. The digraph obtained by identifying the equivalent positions of a digraph is called \textit{identification digraph}.

The main computational and theoretical tool introduced in \cite{1} is the simplified structure diagram of a game. Using this tool, the authors proved many important results and studied the nim-values for cyclic groups, dihedral gruops, abelian groups, symmetric groups and alternating groups.

In this paper, we have used the method and terminology developed by Ernst and Sieben \cite{2} and studied both the impartial games. In section \ref{sec2}, we have given some preliminaries. In subsection \ref{sub1}, we have defined the impartial games on finite groups and their nim-values. In subsections \ref{sub2} and \ref{sub3}, we have defined the achievement game and the avoidance game on finite groups. Some important results are also reproduced in these subsections. In section \ref{sec3} and \ref{sec4}, we have studied the the achievement game and the avoidance game on Frobenius groups and non-abelian group with all abelian subgroups respectively.

Throughout the paper, $\mathbb{Z}_{n}$ will denote the cyclic group of order $n$ under the binary operation of addition modulo $n$, $\mathbb{Z}_{1}$ will denote the trivial group and $\gcd$ will denote the greatest common divisor. All the groups considered in this paper are finite groups. $\mathcal{P}(X)$ will denote the power set of the set $X$.

	\section{Preliminaries}\label{sec2}
	In this section, we will recall some of the definitions and results that we will require later.
	\begin{definition}
		An \textit{Impartial-game} is a combinatorial game in which the moves available for a given position do not depend on whose turn it is $i.e.$ a game with a finite set $X$ of positions along with a starting position and a collection $\{Opt(P)\subseteq X \mid P\in X\}$ of option sets is said to be impartial if for a given position $P$  both the players have same option set $Opt(P)$. A position where the game ends is called a terminal position.  \end{definition}
	
	One such game is \textit{game of nim} which is played with $k$ heaps of stones. To play this game players alternatively take one or more stones from single heap, the player who takes last stone wins or loses the game. 
	
	\begin{definition} The minimum excludant, $mex(A)$ of a set $A$ of ordinals is the smallest ordinal not contained in the set $A$. 
	\end{definition}
	The nim-value, nim($P$) of a position $P$ is the minimum excludant of the set of nim-values of the options of $P$ that is, \[nim(P)=mex(\{nim(Q)\mid Q\in Opt(P)\}).\]
	
	\begin{theorem}(Sprague-Grundy Theorem) \cite[Theorem 1.3, p.180]{3} Every short impartial game is equal to a nim-heap.  \end{theorem} 
	
	The game is an \textit{N-position} if the next player wins and it is a \textit{P-position} if the previous player wins.
	
	\begin{proposition} \cite[Theorem 1.12, p. 56]{3} For a game $G$, a position $Q$ is a \emph{P-position} if and only if $nim(Q)=*0$. \end{proposition}
	
	A subset $S\subseteq G$ is a \textit{generating set} of the group $G$, if the subgroup $\langle S \rangle$ generated by the set $S$ is the group $G$ itself.
	\begin{definition}
		A subgroup $H$ of a group $G$ is called \textit{maximal subgroup} if $\langle H \cup\{x\} \rangle = G$ for any $x\in G\setminus H$. Intersection of all maximal subgroups of $G$ is called the  \textit{Frattini subgroup} of $G$ and is denoted by $\Phi(G)$. \end{definition}
	
	\begin{proposition} Any non-trivial finite group has at least one maximal subgroup.   \end{proposition}
	
	\begin{proposition}\cite[Proposition 2.2, p. 512]{1}
		A subset $S$ of a finite group $G$ is a generating set if and only if $S$ is not contained in any maximal subgroup of $G$.
	\end{proposition}
	
	\subsection{Impartial Games on Finite Groups}\label{sub1}
	In this section, we study two impartial games on finite groups.
	
	\begin{definition} An \textit{Achievement game} is a two-player game played on a group $G$ with the following rules:
		\begin{itemize}
			\item[$(i)$] first player chooses an element $x_1\in G$;
			\item[$(ii)$] at the $r^{th}$ turn, concerned player chooses  $x_r\in G\setminus\{x_1,x_{2}, \cdots, x_{r-1}\}$.
		\end{itemize} 
		The game ends right after the $n^{th}$ turn if $n$ is the
		smallest positive integer such that a player chooses an element $x_{n} \in G\setminus \{x_1, x_2,\cdots, x_{n-1}\}$ and the set $\{x_1, x_{2}, \cdots, x_n\}$ generates the group $G$. The player with the last move wins the game. The game is denoted by $GEN(G)$. A position in $GEN(G)$ is the set of jointly chosen elements $\{x_1, x_2,\cdots, x_r\}$. 
	\end{definition}
	
	In the group $\mathbb{Z}_5= \{0,1,2,3,4\}$, if the first player selects any element from set \{1,2,3,4\}, then the player wins the game and if the first player selects 0, then the second player wins the game.
	
	\begin{definition}An \textit{Avoidance game} is a two-player game played on a group $G$ with the following rules:
		\begin{itemize}
			\item[$(i)$] first player chooses an element $x_1\in G$ such that $\langle x_1\rangle\neq G$;
			\item[$(ii)$] at the $r^{th}$ turn, concerned player chooses  $x_r\in G\setminus\{x_1,x_{2}, \cdots, x_{r-1}\}$ such that $\langle x_1,...,x_{r} \rangle\neq G$.
		\end{itemize} 
		The game ends right after the $n^{th}$ turn if $n$ is the
		smallest positive integer such that a player is not able to choose an element $x_{n} \in G\setminus \{x_1, x_2, \cdots, x_{n-1}\}$ and the set $\{x_1, x_{2}, \cdots, x_n\}$ do not generate the group $G$ but the set $\{x_1,x_2, \cdots,x_n\} \cup \{x\}$ generates the group $G$ for any $x\in G\setminus \{x_1, x_2, \cdots,x_n\}$. The player with the last move wins the game. The game is denoted by $DGN(G)$. A position in $DGN(G)$ is the set of jointly chosen elements $\{x_1, x_2,\cdots, x_r\}$ which must not generate the group $G$. 
	\end{definition}
	
	In the group $\mathbb{Z}_7 = \{0,1,2,3,4,5,6\}$, if the first player selects any element from $\{1,2,3,4,5,6\}$, then the player loses the game and if the first player selects $0$, then the second player loses the game.
	
	\begin{definition}
		Let $P$ and $Q$ be two non-empty subsets of a group $G$. Then $P$ and $Q$ are said to be \textit{automophism equivalent} if there is an automophism $\phi$ of the group $G$ such that $\phi(P)=Q$.
	\end{definition}
	\begin{definition}\cite[p. 513]{1}
		Let $\Gamma$ be a game. A \textit{game digraph} is a diagrammatic representation of the game. For each position $P$, the $Opt(P)$ can be partitioned into automorphism equivalence classes. By removing all but one representative from each of the class, the obtained game-digraph is called \textit{representative game digraph}. \end{definition}
	
	Note that a group automorphism induces an automorphism of the game digraph. Therefore, if $P$ and $Q$ are two \emph{automorphism equivalent} positions of game $GEN(G)$ or $DNG(G)$, then $nim(P) = nim(Q)$. Also, an $r$th position $P = \{x_{1}, x_{2}, \cdots, x_{r}\}$ is called an even position(odd position) accordingly as $r$ is even(odd).
	\subsection{Avoidance game}\label{sub2}
	To study the the avoidance game on a finite group $G$, we first discuss the \emph{nim-value} of each position of the game.
	
	\begin{theorem} \cite[Proposition 3.4, p.5]{1} Let $G$ be a finite group. Then the positions of $DNG(G)$ are subsets of maximal subgroups of the group $G$ and the terminal positions are maximal subgroups of the group $G$. 
	\end{theorem}
	We have the set of all positions of $DNG(G)$. Now we partition this set into a class of sets to simplify our calculations.
	
	Let $\mathcal{M}$ be the set of all maximal subgroups of $G$ and $\mathcal{I}=\{\cap M\mid \emptyset\neq M\subseteq \mathcal{M}\}$ be the set of all possible intersections of maximal subgroups of $G$. The set $\mathcal{I}$ of intersection of subgroups is partially ordered by inclusion. To denote some certain subsets of $\mathcal{I}$, we use the interval notation that is, for $I \in \mathcal{I}$, $(-\infty, I) = \{J\in \mathcal{I} \mid J \subsetneq I\}$. 
	
	\begin{definition} For each $I\in \mathcal{I}$, a structure class $X_{I}$ is defined as the collection of those subsets of $I$ that are not contained in any proper subgroup of $I$ in $\mathcal{I}$, that is,
		\[X_{I} = \mathcal{P}(I)\setminus \cup{\{\mathcal{P}(J)\mid J\in (-\infty, I)\}}\].
	\end{definition}
	
	Let $\mathcal{X} = \{X_{I} \mid I\in \mathcal{I}\}$ be the collection of all structure classes. It is obvious that the Frattini subgroup $\Phi(G)$ is in \(\mathcal{I}\) and $\emptyset\in X_{\Phi(G)}$.
	
	For a structure class $X_I$, the \emph{parity}\footnote{For a set $S$,  the \emph{parity}, $pty(S)$ is defined as the parity of its order.} of the structure class is defined as the \emph{parity} of the corresponding subgroup $I$. A structure class $X_{I}$ is terminal if $I$ is the terminal position. Note that the set $\mathcal{X}$ of all structure classes forms a partition of the set of all positions of the game $DNG(G)$ (see \cite[Corollary 3.9, p. 5]{1}). Next, we see that there is a relation between the elements of given structure class.
	
	The partition $\mathcal{X}$ is compatible with the option relationship between game positions, that is, if $X_{I}, X_{J}\in \mathcal{X}$ are two distinct structure classes and $P$, $Q \in X_{I}$ are two positions, then $Opt(P )\cap X_{J} \ne \emptyset$ if and only if $Opt(Q)\cap X_{J} \ne \emptyset$ (see \cite[Corollary 3.11, p. 6]{1}). Let $X_{I}, X_{J}\in \mathcal{X}$ be two structure classes. Then, if $Opt(I)\cap X_{J} \ne \emptyset$, then $X_{J}$ is said to be an option of $X_{I}$ and is written as $X_{J}\in Opt(X_{I})$. The set $\{X_{I}\mid I\in \mathcal{I}\}$ of the structure classes is called vertex set and the set $\{(X_{I}, X_{J}) \mid X_{J}\in Opt(X_{I})\}$ is called the edge set of the digraph (called the \textit{structure digraph} of the game). 
	Ernst and Sieben \cite{1} proved that  with each structure class only two \emph{nim-values} are associated. They proved that if $P$ and $Q$ are two positions in a structure class $X_I$ such that $pty(P)=pty(Q)$, then $nim(P)=nim(Q)$ (see \cite[Proposition 3.15, pp. 6]{1}). The position(odd or even) of a structure class $X_{I}$ is defined as the position(odd or even) of the corresponding subgroup $I$ in the game digraph.

	In a \textit{structure diagram}, a structure class $X_{I}$ is represented by a triangle pointing down or up corresponding to $pty(I)$ is odd or even respectively. The triangles are divided into two parts, a smaller triangle and a trapezoidial part, where the smaller triangle represents the odd positions of $X_{I}$ and the trapezoid represents the even positions of $X_{I}$. The numbers in the smaller triangle and the trapezoid are the nim-values of these positions. Also, if $X_{J}\in Opt(X_{I})$, then there is a directed edge from $X_{I}$ to $X_{J}$.
	
	The $type$ of a structure class $X_{I}$ is the triple
	\[type(X_{I}) = \left(pty(I),nim(P),nim(Q)\right),\]
	where $P, Q \in X_{I}$ with $pty(P) = 0$ and $pty(Q) = 1$. The option type of $X_{I}$ is the set
	\[otype(X_{I}) := \{type(X_{J}) \mid X_{J} \in Opt(X_{I})\}\]
	and the full option type of $X_{I}$ is the set
	\[Otype(X_{I}) := otype(X_{I}) \cup \{type(X_{I})\}.\] 
	Two structure classes $X_{I}$ and $X_{J}$ are said to be type equivalent if $type(X_{I}) = type(X_{J})$ and $Otype(X_{I}) = Otype(X_{J})$. 
	\begin{figure}[h]
		\centering
		\begin{tikzpicture}
		\node (a) at (0,0) {$X_{I}$};
		\node (b) at (1.5,0) {\begin{tikzpicture}
			\draw[thick,-] (0,0) -- (1,0);
			\draw[thick, -] (0,0) -- (.5,1);
			\draw[thick, -] (1,0) -- (.5,1);
			\path[fill=gray!20, thick] (0,0) -- (.5,1) -- (1,0) -- cycle;
			\draw[thick, -] (0.2,.45) -- (.8,.45) node[pos = .5, above]{$b$};
			\draw[thick, -] (0.2,.45) -- (.8,.45) node[pos = .5, below]{$a$};
			\end{tikzpicture}};
		\node (c) at (3,0) {\begin{tikzpicture}
			\draw[thick,-] (0,0) -- (1,0);
			\draw[thick, -] (0,0) -- (.5,-1);
			\draw[thick, -] (1,0) -- (.5,-1);
			\path[fill=gray!20, thick] (0,0) -- (.5,-1) -- (1,0) -- cycle;
			\draw[thick, -] (0.2,-.45) -- (.8,-.45) node[pos = .5, above]{$a$};
			\draw[thick, -] (0.2,-.45) -- (.8,-.45) node[pos = .5, below]{$b$};
			\end{tikzpicture}};
		\node (d) at (0,-.95) {$type(X_{I})$};
		\node (e) at (1.5,-.95) {$(0,a,b)$};
		\node (f) at (3,-.95) {$(1,a,b)$};
	\end{tikzpicture}
	\caption{Visualization of structure classes and their corresponding types}
\end{figure}
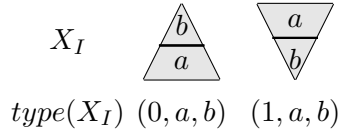

To compute the $types$ of structure classes, we first obtain the structure digraph. Then the $types$ of the structure classes can be computed recursively from the bottom up process using the formulas $type(X_{I}) = (pty(I), a, b)$, where
\begin{center}
$A = \{a^*\mid (\theta,a^*,b^*)\in otype(X_I)\}$ and $B = \{b^*\mid (\theta,a^*,b^*)\in otype(X_I)\}$,\\
$a = mex(B), b = mex(A\cup \{a\})$, if $pty(I) = 0$,\\
$b = mex(A), a = mex(B\cup \{b\})$, if $pty(I) = 1$.
\end{center}
The nim-value of the game is same as the nim-value of the initial position $\emptyset$, which is an even subset of $\Phi(G)$. Because of this, the nim-value of the game is the second component of $type(X_{\Phi(G)})$, which corresponds to the trapezoidal part of the triangle representing the source vertex $X_{\Phi(G)}$ of the structure diagram. Note that if $X_{I}$ is terminal, then $type(X_I)$ must be either $(0, 0, 1)$ or $(1, 1, 0)$ depending on the parity of $X_I$.

A \textit{simplified structure diagram} of $DNG(G)$ is obtained from the structure diagram by identifying two structure classes that are $type$ equivalent. 

Let us determine the $type$ of structure classes in the avoidance games of the group $\mathbb{Z}_6$.
\begin{example}
Let $G= \mathbb{Z}_{6}$. Then $\mathcal{M} = \{\langle2\rangle, \langle3\rangle\}$ and $\mathcal{I}=\{\{0\}, \langle2\rangle, \langle3\rangle\}$. We will start the process from the terminal classes $X_{\langle2\rangle}$ and $X_{\langle3\rangle}$. Since, $X_{\langle2\rangle}$ is terminal and $pty(X_{\langle2\rangle}) = 1$, $type(X_{\langle2\rangle}) = (1,1,0)$. Also, $X_{\langle3\rangle}$ is terminal and $pty(X_{\langle3\rangle}) = 0$, $type(X_{\langle3\rangle}) = (0,0,1)$. For structure class $X_{\{0\}}$, we have $Opt(X_{\{0\}})=\{X_{\langle2\rangle},X_{\langle3\rangle}\}$. Therefore, $otype(X_{\{0\}}) = \{(1,1,0), (0,0,1)\}$. Here, the set $A = \{0,1\}$ and so is the set $B$. Since $pty(X_{\{0\}})=1$, $b = mex(A) = mex(\{0,1\}) = 2$ and $a = mex(B \cup \{b\}) = mex(\{0,1,2\}) = 3$. Hence the $type\left(X_{\{0\}}\right) = (1,3,2)$.
\end{example}
\begin{figure}[h]
\begin{tikzpicture}
\node (1) at (0,0) {$\emptyset_{*3}$};
\node (2) at (-2.25,-1.25) {$\{2\}_{*0}$} edge [<-] (1);
\node (3) at (-0.75,-1.25) {$\{4\}_{*0}$} edge [<-,dashed] (1);
\node (4) at (0.75,-1.25)   {$\{0\}_{*2}$} edge [<-] (1);
\node (5) at (2.25,-1.25)   {$\{3\}_{*1}$} edge [<-] (1);
\node (6) at (-2.25,-2.5) {$\{2,4\}_{*1}$} edge[<-] (3) edge[<-] (2) ;
\node (7) at (-0.75,-2.5) {$\{2,0\}_{*1}$} edge [<-] (2) edge[<-] (4);
\node (8) at (0.75,-2.5)   {$\{0,4\}_{*1}$} edge [<-,dashed] (3)  edge [<-] (4);
\node (9) at (2,-3.75) {$\{0,3\}_{*0}$} edge[<-] (4) edge[<-] (5);
\node(10) at (-1,-3.75) {$\{0,2,4\} _{*0}$} edge[<-](6) edge[<-](7) edge[<-](8);

\node (11) at (7,0) {$\emptyset_{*3}$};
\node (12) at (5,-1.25) {$\{2\}_{*0}$} edge [<-] (11);
\node (14) at (7,-1.25)   {$\{0\}_{*2}$} edge [<-] (11);
\node (15) at (9,-1.25)   {$\{3\}_{*1}$} edge [<-] (11);
\node (16) at (5,-2.5) {$\{2,4\}_{*1}$}  edge[<-] (12) ;
\node (17) at (7,-2.5) {$\{2,0\}_{*1}$} edge [<-] (12) edge[<-] (14);
\node (19) at (8.5,-3.75) {$\{0,3\}_{*0}$} edge[<-] (14) edge[<-] (15);
\node(20) at (6,-3.75) {$\{0,2,4\} _{*0}$} edge[<-](16) edge[<-](17);
\end{tikzpicture}
\caption{Game digraph and representative game digraph for $DNG(\mathbb{Z}_6)$}

\begin{tikzpicture}
\node (1) at (0,0) {$\underset{(1,3,2)}{X_{\Phi(\mathbb{Z}_6)}}$};
\node (2) at (-1.5,-1.5) {$\underset{(1,1,0)}{X_{\langle 2\rangle}}$} edge [<-] (1);
\node (3) at (1.5,-1.5)  {$\underset{(0,0,1)}{X_{\langle 3\rangle}}$}  edge [<-] (1);

\node (a) at (7.5,0)	{\begin{tikzpicture}
	\draw[thick,-] (0,0) -- (1,0);
	\draw[thick, -] (0,0) -- (.5,-1);
	\draw[thick, -] (1,0) -- (.5,-1);
	\path[fill=gray!20, thick] (0,0) -- (.5,-1) -- (1,0) -- cycle;
	\draw[thick, -] (0.2,-.45) -- (.8,-.45) node[pos = .5, above]{$3$};
	\draw[thick, -] (0.2,-.45) -- (.8,-.45) node[pos = .5, below]{$2$};
	\end{tikzpicture}};
\node (b) at (9,-1.5) {\begin{tikzpicture}
	\draw[thick,-] (0,0) -- (1,0);
	\draw[thick, -] (0,0) -- (.5,1);
	\draw[thick, -] (1,0) -- (.5,1);
	\path[fill=gray!20, thick] (0,0) -- (.5,1) -- (1,0) -- cycle;
	\draw[thick, -] (0.2,.45) -- (.8,.45) node[pos = .5, above]{$1$};
	\draw[thick, -] (0.2,.45) -- (.8,.45) node[pos = .5, below]{$0$};
	\end{tikzpicture}};
\node (c) at (6,-1.5)	{\begin{tikzpicture}
	\draw[thick,-] (0,0) -- (1,0);
	\draw[thick, -] (0,0) -- (.5,-1);
	\draw[thick, -] (1,0) -- (.5,-1);
	\path[fill=gray!20, thick] (0,0) -- (.5,-1) -- (1,0) -- cycle;
	\draw[thick, -] (0.2,-.45) -- (.8,-.45) node[pos = .5, above]{${1}$};
	\draw[thick, -] (0.2,-.45) -- (.8,-.45) node[pos = .5, below]{$0$};
	\end{tikzpicture}};
\draw[->] (a) -- (b);
\draw[->] (a) -- (c);
\end{tikzpicture}
\caption{Structure digraph and structure diagram for $DNG(\mathbb{Z}_6)$ with \emph{type} of each structure class.}
\end{figure}

\begin{proposition}  \cite[Proposition 3.20, p. 10]{1} For game $DNG(G)$, \emph{type} of a structure class lies in the set $\{(0,0,1),(1,0,1), (1,1,0),(1,3,2)\}$.  \end{proposition}

\begin{corollary}  \cite[proposition 3.21, p. 10]{1} For game $DNG(G)$, possible \emph{nim-values} are $*0,*1$, or $*3$. \end{corollary}

\begin{proposition} \cite[Proposition 3.22, p. 11]{1} If $G$ is non-trivial group of odd order, then $DNG(G)=*1$.  
\end{proposition}

\begin{proposition} \cite[Proposition 3.23, p. 11]{1} If the Frattini subgroup of a non-trivial group $G$ is of even order, then $DNG(G)=*0$.  
\end{proposition}

\subsection{Achievement game}\label{sub3}
In this section, we study the achievement game on a finite group $G$. For this, an additional structure class $X_G$ is included which contains the terminal positions as those subsets $S$ of $G$ such that $S$ generates $G$ while $S\setminus \{s\}$ does not, for some $s \in S$. Note that this is a slightly abusive notation because $X_G$ does not always contain $G$. For nontrivial groups, the positions of $GEN(G)$ are the positions of $DNG(G)$ together with the elements of $X_G$. If $G$ is the trivial group, then $\Phi(G) = G$ is not a game position of $GEN(G)$ and $X_G = \{\emptyset\}$ is the only structure class. The following is immediate.

The set $\mathcal{Y}=\mathcal{X}\cup\{X_G\}$ forms a partition of game positions of the game $GEN(G)$. Likewise in the game $DNG(G)$, the partition $\mathcal{Y}$ is also compatible with the option relationship between game positions (see \cite[Corollary 4.3, p. 11]{1}). Ernst and Sieben \cite{1} also proved that with each structure class only two \textit{nim-values} are associated.

For the game $GEN(G)$, given a structure class $X_{I}$, $type(X_{I})$, $otype(X_{I})$, $Otype(X_{I})$ and structure digraph are defined same as the game $DNG(G)$ along with the $type$ of terminal structure class $X_G$ is defined to be $(pty(G),0,0)$.

\begin{definition} Let $X_I\in\mathcal{Y}$ be a structure class in $GEN(G)$. Then $X_I$ is called a
\begin{itemize}
\item[$(i)$] terminal structure class if it consists of terminal positions,
\item[$(ii)$] semi-terminal structure class if terminal structure class is an option of $X_I$,
\item[$(iii)$] non-terminal structure class if $X_I$ is neither terminal nor semi-terminal.
\end{itemize}
\end{definition}

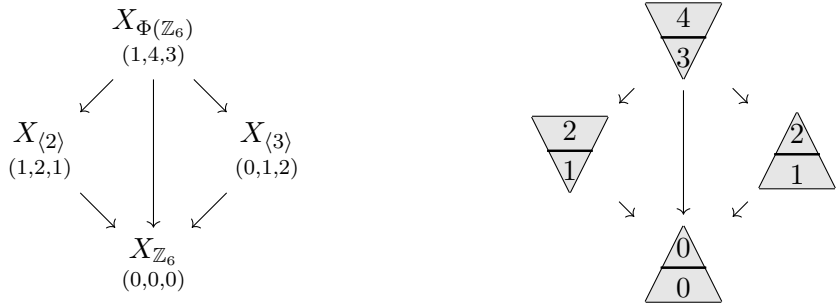
\begin{figure}[h]
\begin{tikzpicture}
\node (1) at (1,0) {$\underset{(1,4,3)}{X_{\Phi(\mathbb{Z}_6)}}$};
\node (2) at (-0.5,-1.5) {$\underset{(1,2,1)}{X_{\langle 2\rangle}}$} edge [<-] (1);
\node (3) at (2.5,-1.5)  {$\underset{(0,1,2)}{X_{\langle 3\rangle}}$}  edge [<-] (1);
\node (4) at (1,-3)  {$\underset{(0,0,0)}{X_{\mathbb{Z}_6}}$}  edge [<-] (1)    edge [<-] (2)   edge [<-] (3);

\node (a) at (8,0)	{\begin{tikzpicture}
\draw[thick,-] (0,0) -- (1,0);
\draw[thick, -] (0,0) -- (.5,-1);
\draw[thick, -] (1,0) -- (.5,-1);
\path[fill=gray!20, thick] (0,0) -- (.5,-1) -- (1,0) -- cycle;
\draw[thick, -] (0.2,-.45) -- (.8,-.45) node[pos = .5, above]{$4$};
\draw[thick, -] (0.2,-.45) -- (.8,-.45) node[pos = .5, below]{$3$};
\end{tikzpicture}};
\node (b) at (9.5,-1.5) {\begin{tikzpicture}
\draw[thick,-] (0,0) -- (1,0);
\draw[thick, -] (0,0) -- (.5,1);
\draw[thick, -] (1,0) -- (.5,1);
\path[fill=gray!20, thick] (0,0) -- (.5,1) -- (1,0) -- cycle;
\draw[thick, -] (0.2,.45) -- (.8,.45) node[pos = .5, above]{$2$};
\draw[thick, -] (0.2,.45) -- (.8,.45) node[pos = .5, below]{$1$};
\end{tikzpicture}};

\node (c) at (6.5,-1.5)	{\begin{tikzpicture}
\draw[thick,-] (0,0) -- (1,0);
\draw[thick, -] (0,0) -- (.5,-1);
\draw[thick, -] (1,0) -- (.5,-1);
\path[fill=gray!20, thick] (0,0) -- (.5,-1) -- (1,0) -- cycle;
\draw[thick, -] (0.2,-.45) -- (.8,-.45) node[pos = .5, above]{${2}$};
\draw[thick, -] (0.2,-.45) -- (.8,-.45) node[pos = .5, below]{$1$};
\end{tikzpicture}};

\node (d) at (8,-3) {\begin{tikzpicture}
\draw[thick,-] (0,0) -- (1,0);
\draw[thick, -] (0,0) -- (.5,1);
\draw[thick, -] (1,0) -- (.5,1);
\path[fill=gray!20, thick] (0,0) -- (.5,1) -- (1,0) -- cycle;
\draw[thick, -] (0.2,.45) -- (.8,.45) node[pos = .5, above]{$0$};
\draw[thick, -] (0.2,.45) -- (.8,.45) node[pos = .5, below]{$0$};
\end{tikzpicture}};
\draw[->] (a) -- (b);
\draw[->] (a) -- (c);
\draw[->] (a) -- (d);
\draw[->] (b) -- (d);
\draw[->] (c) -- (d);
\end{tikzpicture}
\caption{Structure digraph and structure diagram for $GEN(\mathbb{Z}_6)$}
\end{figure} 

Note that, a non-terminal structure class cannot be the option of a semi-terminal structure class.

\begin{proposition} \cite[Corollary 4.8, p. 13]{1} Let $G$ be a non-trivial group of odd order. Then the nim-value of game $GEN(G)$ is either $*1$ or $*2$.  \end{proposition}

\begin{proposition} \label{p4.4}
Let $G$ be a group of even order. If parity of a structure class $X_I$ in $GEN(G)$ is 0, then $type(X_I)\in\{(0,0,0),(0,1,2),(0,0,2),(0,0,1)\}$. \end{proposition}
\begin{proof} Let $G$ be a group of even order and $t_0=(0,0,0),t_1=(0,1,2),t_2=(0,0,2)$ and $t_3=(0,0,1)$. Then $type$ of the terminal structure class $X_{G}$ is $(0,0,0)$.

Let $X_I$ be a semi-terminal structure class of parity 0. We will show that $type(X_I)=(0,1,2)$. Note that options of $X_{I}$ are either terminal classes or both terminal classes and semi-terminal classes. If option of $X_{I}$ is a terminal class, then $otype(X_{I}) = \{t_{0}\}$ and so, the $type(X_{I}) = t_{1}$. If options of $X_{I}$ are both terminal and semi-terminal classes, then $otype(X_{I}) = \{t_{0}, t_{1}\}$ and so, the $type(X_{I}) = t_{1}$. Thus in both the cases, $type(X_{I}) = t_{1}$.

Let $X_I$ be a non-terminal structure class of parity 0. We will show that $type(X_I)\in\{(0,0,1),(0,0,2)\}$. Note that terminal class cannot be an option of $X_{I}$. So, $t_{0} \not\in otype(X_{I})$. Since the options of non-terminal classes are either semi-terminal classes or non-terminal classes or both semi-terminal classes and non-terminal classes, $otype(X_I)\subseteq \{t_{1}, t_{2}, t_{3}\}$. The following table gives the $otype$ and $type$ of non-terminal class $X_{I}$ 

\begin{table}[h!]
\centering
\begin{tabular}{|c|c|c|c|c|c|c|c|}
\hline
\emph{otype} & $\{t_1\}$          &$\{t_2\}$           &$\{t_3\}$            &$\{t_1,t_2\}$              & $\{t_1,t_3\}$        &$\{t_2,t_3\}$      &$\{t_1,t_2,t_3\}$ \\
\hline
$type$      &$t_2$                &$t_3$                  &$t_3$                   &$t_2$                         &$t_2$                      &$t_3$                      &$t_2$  \\
\hline
\end{tabular}
\end{table} 
Hence $type(X_I)\in\{(0,0,0),(0,1,2),(0,0,2),(0,0,1)\}$, where $I$ is a subgroup of even order.
\end{proof}

\section{Frobenius Groups}\label{sec3}
In this section, we study the achievement and the avoidance games on Frobenius groups $F_{p}$ with the following presentation 
\[\langle a,b: a^p=b^{p-1}=1, b^{-1}ab=a^n, \gcd(n,p)=1\rangle,\] 
where $p\geq5$ is a prime.  Note that, $Fp$ is isomorphic to $\mathbb{Z}_p\rtimes\mathbb{Z}_{p-1}$. First, we study the maximal subgroups and the Frattini subgroup of the group $F_{p}$. Throughout this section, $F_p$ denotes the Frobenius group the above form. 

\begin{theorem}  \cite[Theorem 5.6, p. 29]{4}\label{fmax}
Let $F_p$ be the Frobenius group of order $p(p-1)$ such that $p-1=p_1^{a_1}p_2^{a_2}\cdots p_k^{a_k}$, where $p_1,p_2,\cdots, p_k$ are distinct prime numbers such that $p_1=2$ and $a_1,a_2,\cdots,a_k$ are non-negative integers. Then maximal subgroups are divided into the following conjugacy classes 
\begin{itemize}
\item[$(i)$] there are $p$ maximal, conjugate, cyclic subgroups of order $p-1$,
\item[$(ii)$]    there are $k$ maximal subgroups isomorphic to $\mathbb{Z}_p\rtimes\mathbb{Z}_{\frac{p-1}{p_i}}$, where $i\in\{1,2,\cdots,k\}$. 
\end{itemize}
\end{theorem} 

\begin{corollary}
	 Frattini subgroup of the Frobenius group $F_p$ is the trivial group. 
\end{corollary}

\begin{proposition} DGN$(F_p)=*0$.    \end{proposition}
\begin{proof} Let $p-1=p_1^{a_1}p_2^{a_2}\cdots p_k^{a_k}$, where $p_1,p_2,\cdots,p_k$ are primes such that $p_1=2$ and $a_1,a_2,\cdots,a_k$ are non-negative integers. Then, using the Theorem \ref{fmax}, the set $\mathcal{I}$ of intersection of maximal subgroups of $F_{p}$ is
\[\mathcal{I} = \{\mathbb{Z}_{1}, H_{1}, H_{2}, H_{3}, H_{4}, H_{5}\},\]
where $H_{1} \simeq \mathbb{Z}_{p-1}, H_{2}\simeq \mathbb{Z}_p\rtimes\mathbb{Z}_{\frac{p-1}{p_{i}}}, H_{3}\simeq \mathbb{Z}_{\frac{p-1}{p_{i}}}, H_{4} \simeq \mathbb{Z}_p\rtimes\mathbb{Z}_{\frac{p-1}{\pi}}, H_{5} \simeq \mathbb{Z}_{\frac{p-1}{\pi}}$ and $\pi$ is the product of atleast two primes $p_{i} (1\le i \le k)$. Now let us determine the $type$ of each structure class in two cases namely when $4$ divides $(p-1)$ and when $4$ does not divide $(p-1)$.

\textit{Case$(i)$.} Let $a_{1} = 1$. Then $p-1=p_1p_2^{a_2}\cdots p_k^{a_k}$. Let us start with the structure classes associated with the maximal subgroups. Since the maximal subgroups $\mathbb{Z}_{p-1}$ and $\mathbb{Z}_p \rtimes \mathbb{Z}_{\frac{p-1}{p_j}}$, where $j\in \{2, 3, \cdots, k\}$ are of even order, the corresponding structure classes are of $type$ $(0,0,1)$. Also the maximal subgroup $\mathbb{Z}_p\rtimes\mathbb{Z}_{\frac{p-1}{2}}$ is of odd order. Therefore, $type$ of the corresponding structure class is $(1,1,0)$. 

Now let us find the $type$ of structure classes associated with elements of $\mathcal{I}$ that are not maximal subgroups. Let $I = \mathbb{Z}_p\rtimes \mathbb{Z}_{\frac{p-1}{\pi}}$. Now, if $pty(I) = 0$, then all the options of $X_{I}$ are of parity 0. Therefore, $otype(X_{I}) = \{(0, 0, 1)\}$ and so, $type(X_{I}) = (0,0,1)$. If $pty(I) = 1$, then 
\begin{align*}
\text{either}\; &Opt(X_{I}) = L_{1} = \left\{X_{\mathbb{Z}_p\rtimes\mathbb{Z}_{\frac{p-1}{2}}}, X_{\mathbb{Z}_p \rtimes \mathbb{Z}_{\frac{p-1}{p_j}}}\right\}, \text{where}\; j\in \{2, 3, \cdots, k\}\\
\text{or}\; &Opt(X_{I}) = L_{2} = L_{1} \bigcup  \left\{X_{\mathbb{Z}_p\rtimes \mathbb{Z}_{\frac{p-1}{\pi^{\prime}}}}\mid \pi \; \text{divides}\; \pi^{\prime}\right\}.
\end{align*}

Now, we determine the $types$ of structure classes associated with subgroups $I = \mathbb{Z}_{\frac{p-1}{\pi}}$. If $pty(I) = 0$, then $Opt(X_{I})$ will consist of only the even order subgroups. Therefore, $otype(X_{I}) = \{(0,0,1)\}$ and so, $type(X_{I}) = (0,0,1)$. Now, let $pty(I) = 1$. Then for $j\in \{2, 3, \cdots, k\}$

\begin{align*}
\text{either}\; &Opt(X_{I}) = J_{1} = \left\{X_{\mathbb{Z}_{p-1}}, X_{\mathbb{Z}_{\frac{p-1}{2}}}, X_{\mathbb{Z}_{2\frac{p-1}{\pi}}}, X_{\mathbb{Z}_p \rtimes I}\right\}\\
\text{or}\; &Opt(X_{I}) = J_{2} = J_{1} \bigcup \left\{X_{\mathbb{Z}_{\frac{p-1}{\pi^{\prime}}}} \mid \pi \; \text{divides}\; \pi^{\prime}\right\}.
\end{align*}

By the above discussion, $type(X_{\mathbb{Z}_p \rtimes I}) = (1,3,2)$.
Since $pty\left(\mathbb{Z}_{2\frac{p-1}{\pi}}\right) = 0$, $type\left(X_{\mathbb{Z}_{2\frac{p-1}{\pi}}}\right) = (0,0,1)$. Also, note that $Opt\left(X_{\mathbb{Z}_{\frac{p-1}{2}}}\right) = \left\{X_{\mathbb{Z}_{p-1}}, X_{\mathbb{Z}_{p}\rtimes \mathbb{Z}_{\frac{p-1}{2}}}\right\}$. Therefore, $otype\left(X_{\mathbb{Z}_{\frac{p-1}{2}}}\right) = \{(0,0,1), (1,1,0)\}$ and so, $type\left(X_{\mathbb{Z}_{\frac{p-1}{2}}}\right) = (1,3,2)$. Now it is easy to see that if $Opt(X_{I}) = J_{1}$, then $otype(X_{I}) = \{(1,3,2), (0,0,1)\}$ and so, $type(X_{I}) = (1,0,1)$. By the similar argument, if $Opt(X_{I}) = J_{2}$, then $otype(X_{I}) = \{(1,3,2), (0,0,1), (1,0,1)\}$ and so, $type(X_{I}) = (1, 0, 1)$.

Now, for the structure class $X_{\mathbb{Z}_{1}}$, for $j\in \{2,3, \cdots, k\}$, we have
\[ Opt\left(X_{\mathbb{Z}_{1}}\right) = \left\{X_{\mathbb{Z}_{p-1}}, X_{\mathbb{Z}_{\frac{p-1}{2}}}, X_{\mathbb{Z}_{\frac{p-1}{p_{j}}}}\right\}\bigcup \left\{X_{\mathbb{Z}_{\frac{p-1}{\pi^{\prime}}}}\mid \pi \; \text{divides}\; \pi^{\prime}\right\}.\]
Thus $otype\left(X_{\mathbb{Z}_{1}}\right) = \{(0,0,1), (1,3,2), (1, 0, 1)\}$. Hence in this case $type\left(X_{\mathbb{Z}_{1}}\right) = (1, 0, 1)$.

\textit{Case$(ii)$.} Let $a_1\geq 2$. Then $p-1=p_1^{a_1}p_2^{a_2}\cdots p_k^{a_k}$. In this case, the parity of all elements of the set $\mathcal{I}$ except $\mathbb{Z}_{1}$ is 0. Therefore, $type(X_I)=(0,0,1)$, where $I \in \mathcal{I} \setminus \mathbb{Z}_{1}$. This implies that, for the structure class $X_{\mathbb{Z}_{1}}$, $otype(X_{\mathbb{Z}_{1}}) = \{(0,0,1)\}$. Thus $type\left(X_{\mathbb{Z}_{1}}\right) = (1,0,1)$. 

Therefore, combining both the cases $(i)-(ii)$, we get $type\left(X_{\mathbb{Z}_{1}}\right) = (1,0,1)$. Since $\mathbb{Z}_{1}$ is the Frattini subgroup, $DGN(F_p)= *0$.
\end{proof}
The structure digraph and the simplified structure diagram for the games $DGN(F_{13})$ and $DGN(F_{19})$ are given as below.
\begin{figure}[h!]
\begin{tikzpicture}
\node (a) at (0,0) {$\underset{(1,0,1)}{\mathbb{Z}_{1}}$};
\node (b) at (0,-1.25) {$\underset{(0,0,1)}{\mathbb{Z}_2}$}                                            edge[<-] (a);
\node (c) at (-2,-2.5) {$\underset{(0,0,1)}{\mathbb{Z}_4}$}                                            edge[<-] (a)   edge[<-] (b);
\node (d) at (0,-2.5) {$\underset{(0,0,1)}{\mathbb{Z}_{13}\rtimes\mathbb{Z}_{2}}$}      edge[<-] (b);
\node (e) at (3,-2.5) {$\underset{(0,0,1)}{\mathbb{Z}_6}$}                                            edge[<-] (a)   edge[<-] (b);
\node (f) at (-2,-4) {$\underset{(0,0,1)}{\mathbb{Z}_{13}\rtimes\mathbb{Z}_{4}}$}       edge[<-] (c)   edge[<-] (d);
\node (g) at (1.05,-4) {$\underset{(0,0,1)}{\mathbb{Z}_{13}\rtimes\mathbb{Z}_{6}}$}       edge[<-] (e)  edge[<-] (d);
\node (h) at (3,-4) {$\underset{(0,0,1)}{\mathbb{Z}_{12}}$}                                        edge[<-] (a)   edge[<-] (b)    edge[<-] (c)   edge[<-] (e);

\node (1) at (8,-1)	{\begin{tikzpicture}
\draw[thick,-] (0,0) -- (1,0);
\draw[thick, -] (0,0) -- (.5,-1);
\draw[thick, -] (1,0) -- (.5,-1);
\path[fill=gray!20, thick] (0,0) -- (.5,-1) -- (1,0) -- cycle;
\draw[thick, -] (0.2,-.45) -- (.8,-.45) node[pos = .5, above]{$0$};
\draw[thick, -] (0.2,-.45) -- (.8,-.45) node[pos = .5, below]{$1$};
\end{tikzpicture}};
\node (2) at (8,-3) {\begin{tikzpicture}
\draw[thick,-] (0,0) -- (1,0);
\draw[thick, -] (0,0) -- (.5,1);
\draw[thick, -] (1,0) -- (.5,1);
\path[fill=gray!20, thick] (0,0) -- (.5,1) -- (1,0) -- cycle;
\draw[thick, -] (0.2,.45) -- (.8,.45) node[pos = .5, above]{$1$};
\draw[thick, -] (0.2,.45) -- (.8,.45) node[pos = .5, below]{$0$};
\end{tikzpicture}};

\draw[->] (1) -- (2);

\end{tikzpicture}
\caption{Structure digraph and simplify structure diagram for $DGN(F_{13})$.}
\begin{tikzpicture}
\node (a) at (0,0) {$\underset{(1,0,1)}{\mathbb{Z}_{1}}$};
\node (b) at (0,-1.25) {$\underset{(1,0,1)}{\mathbb{Z}_3}$}                                            edge[<-] (a);
\node (c) at (-2,-2.5) {$\underset{(1,3,2)}{\mathbb{Z}_9}$}                                            edge[<-] (a)   edge[<-] (b);
\node (d) at (0,-2.5) {$\underset{(1,3,2)}{\mathbb{Z}_{19}\rtimes\mathbb{Z}_{3}}$}      edge[<-] (b);
\node (e) at (2.5,-2.5) {$\underset{(0,0,1)}{\mathbb{Z}_6}$}                                            edge[<-] (a)   edge[<-] (b);
\node (f) at (-2,-3.75) {$\underset{(1,1,0)}{\mathbb{Z}_{19}\rtimes\mathbb{Z}_{9}}$}       edge[<-] (c)   edge[<-] (d);
\node (g) at (0,-3.75) {$\underset{(0,0,1)}{\mathbb{Z}_{19}\rtimes\mathbb{Z}_{6}}$}       edge[<-] (e)   edge[<-] (d) ;
\node (h) at (2.5,-3.75) {$\underset{(0,0,1)}{\mathbb{Z}_{18}}$}                                        edge[<-] (a)   edge[<-] (b)    edge[<-] (c)   edge[<-] (e);

\node (1) at (6.75,0.3)	{\begin{tikzpicture}
\draw[thick,-] (0,0) -- (1,0);
\draw[thick, -] (0,0) -- (.5,-1);
\draw[thick, -] (1,0) -- (.5,-1);
\path[fill=gray!20, thick] (0,0) -- (.5,-1) -- (1,0) -- cycle;
\draw[thick, -] (0.2,-.45) -- (.8,-.45) node[pos = .5, above]{$0$};
\draw[thick, -] (0.2,-.45) -- (.8,-.45) node[pos = .5, below]{$1$};
\end{tikzpicture}};
\node (3) at (4.5,-2.1)	{\begin{tikzpicture}
\draw[thick,-] (0,0) -- (1,0);
\draw[thick, -] (0,0) -- (.5,-1);
\draw[thick, -] (1,0) -- (.5,-1);
\path[fill=gray!20, thick] (0,0) -- (.5,-1) -- (1,0) -- cycle;
\draw[thick, -] (0.2,-.45) -- (.8,-.45) node[pos = .5, above]{$3$};
\draw[thick, -] (0.2,-.45) -- (.8,-.45) node[pos = .5, below]{$2$};
\end{tikzpicture}};
\node (2) at (6.75,-1.5)	{\begin{tikzpicture}
\draw[thick,-] (0,0) -- (1,0);
\draw[thick, -] (0,0) -- (.5,-1);
\draw[thick, -] (1,0) -- (.5,-1);
\path[fill=gray!20, thick] (0,0) -- (.5,-1) -- (1,0) -- cycle;
\draw[thick, -] (0.2,-.45) -- (.8,-.45) node[pos = .5, above]{$0$};
\draw[thick, -] (0.2,-.45) -- (.8,-.45) node[pos = .5, below]{$1$};
\end{tikzpicture}};
\node (4) at (9.5,-2.1) {\begin{tikzpicture}
\draw[thick,-] (0,0) -- (1,0);
\draw[thick, -] (0,0) -- (.5,1);
\draw[thick, -] (1,0) -- (.5,1);
\path[fill=gray!20, thick] (0,0) -- (.5,1) -- (1,0) -- cycle;
\draw[thick, -] (0.2,.45) -- (.8,.45) node[pos = .5, above]{$1$};
\draw[thick, -] (0.2,.45) -- (.8,.45) node[pos = .5, below]{$0$};
\end{tikzpicture}};
\node (5) at (6,-3.70)	{\begin{tikzpicture}
\draw[thick,-] (0,0) -- (1,0);
\draw[thick, -] (0,0) -- (.5,-1);
\draw[thick, -] (1,0) -- (.5,-1);
\path[fill=gray!20, thick] (0,0) -- (.5,-1) -- (1,0) -- cycle;
\draw[thick, -] (0.2,-.45) -- (.8,-.45) node[pos = .5, above]{$1$};
\draw[thick, -] (0.2,-.45) -- (.8,-.45) node[pos = .5, below]{$0$};
\end{tikzpicture}};
\node (6) at (8.5,-3.70) {\begin{tikzpicture}
\draw[thick,-] (0,0) -- (1,0);
\draw[thick, -] (0,0) -- (.5,1);
\draw[thick, -] (1,0) -- (.5,1);
\path[fill=gray!20, thick] (0,0) -- (.5,1) -- (1,0) -- cycle;
\draw[thick, -] (0.2,.45) -- (.8,.45) node[pos = .5, above]{$1$};
\draw[thick, -] (0.2,.45) -- (.8,.45) node[pos = .5, below]{$0$};
\end{tikzpicture}};

\draw[->] (1) -- (2);
\draw[->] (1) -- (3);
\draw[->] (1) -- (4);
\draw[->] (1) -- (6);
\draw[->] (2) -- (3);
\draw[->] (2) -- (4);
\draw[->] (2) -- (6);
\draw[->] (3) -- (5);
\draw[->] (3) -- (6);
\draw[->] (4) -- (6);

\end{tikzpicture}
\caption{Structure digraph and simplify structure diagram for $DGN(F_{19})$.}
\end{figure}
\newpage

\begin{proposition} \[ GEN(F_p) =\left\{
\begin{array}{ll}
*0, &\text{if}\; 4\; \text{divides}\; (p-1)\\
*1, &\text{if}\; 4 \; \text{does not divide}\; (p-1) 
\end{array} 
\right.\; . \]   
\end{proposition} 
\begin{proof} 
Let $p-1=p_1^{a_1}p_2^{a_2}\cdots p_k^{a_k}$, where $p_1, p_2, \cdots, p_k$ are primes such that $p_1 =2$ and $a_1,a_2,\cdots,a_k$ are non-negative integers. Then using the Theorem \ref{fmax}, the set $\mathcal{I}$ of intersection of maximal subgroups is
\[\mathcal{I} = \{\mathbb{Z}_{1}, H_{1}, H_{2}, H_{3}, H_{4}, H_{5}\},\]
where $H_{1} \simeq \mathbb{Z}_{p-1}, H_{2}\simeq \mathbb{Z}_p\rtimes\mathbb{Z}_{\frac{p-1}{p_{i}}}, H_{3}\simeq \mathbb{Z}_{\frac{p-1}{p_{i}}}, H_{4} \simeq \mathbb{Z}_p\rtimes\mathbb{Z}_{\frac{p-1}{\pi}}, H_{5} \simeq \mathbb{Z}_{\frac{p-1}{\pi}}$ and $\pi$ is the product of atleast two primes $p_{i} (1\le i \le k)$. The set of all the structure classes for the game $GEN(F_{p})$ is
\[\mathcal{Y}=\left\{X_{\mathbb{Z}_{1}}, X_{H_{1}},X_{H_{2}},X_{H_{3}}, X_{H_{4}}, X_{H_{5}}, X_{F_p}\right\}.\]
Now we determine the $type$ of these structure classes. First, note that $X_{F_p}$ is terminal structure class with parity 0. Therefore, $type(X_{F_p}) = (0,0,0)$. Now let us determine the $type$ of other structure classes in two cases namely when $4$ divides $(p-1)$ and when $4$ does not divide $(p-1)$.

\textit{Case$(i)$.} Let $a_{1} =1$. Then $p-1=p_1p_2^{a_2}\cdots p_k^{a_k}$. Let $I$ be a semi-terminal structure class. Then we have the following
\begin{itemize}
\item[$(i)$] if $I \in \{\mathbb{Z}_{p-1}, \mathbb{Z}_p\rtimes\mathbb{Z}_{\frac{p-1}{p_j}}\}$, then $type(X_{I}) = (0,1,2)$, where $j\in \{2,3, \cdots, k\}$,
\item[$(ii)$] if $I = \mathbb{Z}_p \rtimes \mathbb{Z}_{\frac{p-1}{2}}$, then $type(X_{I}) = (1,2,1)$,
\item[$(iii)$] if $I = \mathbb{Z}_p \rtimes \mathbb{Z}_{\frac{p-1}{\pi}}$, then for $pty(I) = 0$, we have $Opt(X_{I}) =  \left\{X_{F_{p}}, X_{\mathbb{Z}_p \rtimes \mathbb{Z}_{\frac{p-1}{p_j}}}\right\}$, where $j\in \{2,3, \cdots, k\}$. Therefore, $otype(X_{I}) = \{(0,0,0), (0,1,2)\}$ and so, $type(X_{I}) = (0,1,2)$. If $pty(I) = 1$, then $Opt(X_{I}) = \left\{X_{F_p}, X_{\mathbb{Z}_p \rtimes \mathbb{Z}_{\frac{p-1}{p_j}}}, X_{\mathbb{Z}_p \rtimes \mathbb{Z}_{\frac{p-1}{2}}}\right\}$. Therefore, $otype(X_{I}) = \{(0,0,0), (0,1,2), (1,2,1)\}$ and so, $type(X_{I}) = (1, 4, 3)$.

\end{itemize}
Now, we determine the $type$ of all non-terminal structure classes. Note that, the non-terminal structure classes are given by  $I = \mathbb{Z}_{\frac{p-1}{\pi}}$. If $pty(I) = 0$, then for $j\in \{2,3, \cdots,k\}$
\begin{align*}
\text{either}\; &Opt(X_{I}) = I_{1} = \left\{X_{\mathbb{Z}_p \rtimes \mathbb{Z}_{\frac{p-1}{p_{j}}}}, X_{\mathbb{Z}_{p-1}}  \right\}\\
\text{or}\; &Opt(X_{I}) = I_{2} = I_{1} \cup J,
\end{align*}
where $J = \left\{X_{\mathbb{Z}_{p}\rtimes \mathbb{Z}_{\frac{p-1}{\pi^{\prime}}}}, X_{\mathbb{Z}_{\frac{p-1}{\pi^{\prime}}}} \mid \pi\; \text{divides}\; \pi^{\prime}\right\}$ and parity of members of $J$ is 0. If $Opt(X_{I}) = I_{1}$, then $otype(X_{I}) = \{(0,1,2)\}$ and so, $type(X_{I}) = (0, 0, 2)$. Now, if $Opt(X_{I}) = I_{2}$, then $otype(X_{I}) = \{(0,1,2), (0,0,2)\}$ and so, $type(X_{I}) = (0, 0, 2)$. 

Now, let $pty(I) = 1$. Then for $I = \mathbb{Z}_{\frac{p-1}{2}}$, the $Opt(X_{I}) = \left\{X_{\mathbb{Z}_{p}\rtimes \mathbb{Z}_{\frac{p-1}{2}}}, X_{\mathbb{Z}_{p-1}} \right\}$. Therefore, $otype(X_{I}) = \{(1,2,1), (0,1,2)\}$ and so, $type(X_{I}) = (1, 3, 0)$.

Now, for the non-terminal structure class $X_{I}$, where $I \ne  \mathbb{Z}_{\frac{p-1}{2}}$, we have 
\begin{align*}
\text{either}\; & Opt(X_{I}) = J_{1} = \left\{X_{\mathbb{Z}_{p-1}}, X_{\mathbb{Z}_{\frac{p-1}{2}}},X_{\mathbb{Z}_{\frac{p-1}{p_j}}}, X_{\mathbb{Z}_p\rtimes I}\right\}, \; \text{where}\; j\in \{2,3,\cdots,k\}\\
\text{or}\; &Opt(X_{I}) = J_{2} = J_{1} \bigcup \left\{X_{\mathbb{Z}_{\frac{p-1}{\pi^{\prime}}}} \mid \ \pi\; \text{divides}\; \pi^{\prime}\right\}. 
\end{align*} 
Then, for $Opt(X_{I}) = J_{1}$, $otype(X_{I})= \{(0, 1, 2), (1, 3, 0), (0, 0, 2), (1, 4, 3)\}$ and so, $type(X_{I}) = (1, 1, 2)$. If $Opt(X_{I}) = J_{2}$, then $otype(X_{I}) = \{(0, 1, 2), (1, 3, 0), (0, 0, 2), (1, 4, 3), (1,1,2)\}$ and so, $type(X_{I}) = (1, 1, 2)$.

For the structure class $X_{\mathbb{Z}_{1}}$, we have for $j\in \{2,3, \cdots,k\}$
\[ Opt\left(X_{\mathbb{Z}_{1}}\right) = \left\{X_{\mathbb{Z}_{p-1}}, X_{\mathbb{Z}_{\frac{p-1}{2}}}, X_{\mathbb{Z}_{\frac{p-1}{p_{j}}}}\right\}\bigcup \left\{X_{\mathbb{Z}_{\frac{p-1}{\pi^{\prime}}}}\mid \pi \; \text{divides}\; \pi^{\prime}\right\}.\]
Therefore, $otype(X_{I}) = \{(0, 1, 2), (1, 3, 0), (0, 0, 2), (1,1,2)\}$ and so, $type\left(X_{\mathbb{Z}_{1}}\right) = (1, 1, 2)$. Since the Frattini Subgoup of $F_{P}$ is $\mathbb{Z}_{1}$, $GEN(F_{p}) = *1$.

\textit{Case$(ii)$.} Let $a_1\geq2$. Then $p-1=p_1^{a_1}p_2^{a_2}\cdots p_k^{a_k}$ and each semi-terminal structure class is of parity 0. Therefore, the $type$ of each semi-terminal structure class is $(0,1,2)$. Now, let $X_{I}$ be a non-terminal structure class, where $I \ne \{0\}$. Then $pty(X_{I}) = 0$ and the option set of $X_{I}$ is
\begin{align*}
\text{either}\; &Opt(X_{I}) = J_{1} = \left\{X_{\mathbb{Z}_{p}\rtimes I}, X_{\mathbb{Z}_{p-1}}\right\}\\
\text{or}\; &Opt(X_{I}) = J_{2} = J_{1} \bigcup \left\{X_{\mathbb{Z}_{\frac{p-1}{\pi^{\prime}}}} \mid \pi\; \text{divides}\;  \pi^{\prime}\right\}.
\end{align*}
If $Opt(X_{I}) = J_{1}$, then $otype(X_{I}) = \{(0,1,2)\}$ and so, $type(X_{I}) = (0, 0, 2)$. On the other hand, if $Opt(X_{I}) = J_{2}$, then $otype(X_{I}) = \{(0,1,2), (0, 0, 2)\}$ and so, $type(X_{I}) = (0, 0, 2)$. For the structure class $X_{\mathbb{Z}_{1}}$, we have 
\[Opt\left(X_{\mathbb{Z}_{1}}\right)=X_{\mathbb{Z}_{p-1}} \bigcup \left\{X_{\mathbb{Z}_{\frac{p-1}{\pi^{\prime}}}}\mid \pi \; \text{divides}\; \pi^{\prime}\right\}.\]  
Therefore, the $otype(X_{\mathbb{Z}_{1}}) = \{(0, 1, 2), (0, 0, 2)\}$ and so, $type(X_{\mathbb{Z}_{1}}) = (1, 0, 2)$. Since $\mathbb{Z}_{1}$ is the Frattini subgroup of $F_{p}$, $GEN(F_p) = *0$.
\end{proof}
The structure digraph and the simplified structure diagram for the game $GEN(F_{13})$ and $GEN(F_{19})$ are given as below.
\begin{figure}[h!]
\begin{tikzpicture}
\node (a) at (0,0) {$\underset{(1,0,2)}{\mathbb{Z}_{1}}$};
\node (b) at (0,-1.25) {$\underset{(0,0,2)}{\mathbb{Z}_2}$}                                            edge[<-] (a);
\node (c) at (-2,-2.5) {$\underset{(0,0,2)}{\mathbb{Z}_4}$}                                            edge[<-] (a)   edge[<-] (b);
\node (d) at (0,-2.5) {$\underset{(0,1,2)}{\mathbb{Z}_{13}\rtimes\mathbb{Z}_{2}}$}      edge[<-] (b);
\node (e) at (3,-2.5) {$\underset{(0,0,2)}{\mathbb{Z}_6}$}                                            edge[<-] (a)   edge[<-] (b);
\node (f) at (-2,-4) {$\underset{(0,1,2)}{\mathbb{Z}_{13}\rtimes\mathbb{Z}_{4}}$}       edge[<-] (c)   edge[<-] (d);
\node (g) at (1.05,-4) {$\underset{(0,1,2)}{\mathbb{Z}_{13}\rtimes\mathbb{Z}_{6}}$}       edge[<-] (e)  edge[<-] (d);
\node (h) at (3,-4) {$\underset{(0,1,2)}{\mathbb{Z}_{12}}$}                                        edge[<-] (a)   edge[<-] (b)    edge[<-] (c)   edge[<-] (e);
\node (i) at (0,-5.25) {$\underset{(0,0,0)}{\mathbb{Z}_{13}\rtimes\mathbb{Z}_{12}}$}     edge[<-] (f)    edge[<-] (g)     edge[<-] (h)    edge[<-] (d);

\node (1) at (8,-0.25)	{\begin{tikzpicture}
\draw[thick,-] (0,0) -- (1,0);
\draw[thick, -] (0,0) -- (.5,-1);
\draw[thick, -] (1,0) -- (.5,-1);
\path[fill=gray!20, thick] (0,0) -- (.5,-1) -- (1,0) -- cycle;
\draw[thick, -] (0.2,-.45) -- (.8,-.45) node[pos = .5, above]{$0$};
\draw[thick, -] (0.2,-.45) -- (.8,-.45) node[pos = .5, below]{$2$};
\end{tikzpicture}};
\node (2) at (6,-1.5) {\begin{tikzpicture}
\draw[thick,-] (0,0) -- (1,0);
\draw[thick, -] (0,0) -- (.5,1);
\draw[thick, -] (1,0) -- (.5,1);
\path[fill=gray!20, thick] (0,0) -- (.5,1) -- (1,0) -- cycle;
\draw[thick, -] (0.2,.45) -- (.8,.45) node[pos = .5, above]{$2$};
\draw[thick, -] (0.2,.45) -- (.8,.45) node[pos = .5, below]{$0$};
\end{tikzpicture}};
\node (3) at (8,-2.25) {\begin{tikzpicture}
\draw[thick,-] (0,0) -- (1,0);
\draw[thick, -] (0,0) -- (.5,1);
\draw[thick, -] (1,0) -- (.5,1);
\path[fill=gray!20, thick] (0,0) -- (.5,1) -- (1,0) -- cycle;
\draw[thick, -] (0.2,.45) -- (.8,.45) node[pos = .5, above]{$2$};
\draw[thick, -] (0.2,.45) -- (.8,.45) node[pos = .5, below]{$1$};
\end{tikzpicture}};
\node (4) at (8,-4.25) {\begin{tikzpicture}
\draw[thick,-] (0,0) -- (1,0);
\draw[thick, -] (0,0) -- (.5,1);
\draw[thick, -] (1,0) -- (.5,1);
\path[fill=gray!20, thick] (0,0) -- (.5,1) -- (1,0) -- cycle;
\draw[thick, -] (0.2,.45) -- (.8,.45) node[pos = .5, above]{$0$};
\draw[thick, -] (0.2,.45) -- (.8,.45) node[pos = .5, below]{$0$};
\end{tikzpicture}};
\draw[->] (1) -- (2);
\draw[->] (1) -- (3);
\draw[->] (2) -- (3);
\draw[->] (3) -- (4);
\end{tikzpicture}
\caption{Structure digraph and simplify structure diagram for $GEN(F_{13})$.}
\begin{tikzpicture}
\node (a) at (0,0) {$\underset{(1,1,2)}{\mathbb{Z}_{1}}$};
\node (b) at (0,-1.25) {$\underset{(1,1,2)}{\mathbb{Z}_3}$}                                            edge[<-] (a);
\node (c) at (-2,-2.5) {$\underset{(1,3,0)}{\mathbb{Z}_9}$}                                            edge[<-] (a)   edge[<-] (b);
\node (d) at (0,-2.5) {$\underset{(1,4,3)}{\mathbb{Z}_{19}\rtimes\mathbb{Z}_{3}}$}      edge[<-] (b);
\node (e) at (3,-2.5) {$\underset{(0,0,2)}{\mathbb{Z}_6}$}                                            edge[<-] (a)   edge[<-] (b);
\node (f) at (-2,-3.75) {$\underset{(1,2,1)}{\mathbb{Z}_{19}\rtimes\mathbb{Z}_{9}}$}       edge[<-] (c)   edge[<-] (d);
\node (g) at (1,-3.75) {$\underset{(0,1,2)}{\mathbb{Z}_{19}\rtimes\mathbb{Z}_{6}}$}       edge[<-] (e)   edge[<-] (d) ;
\node (h) at (2.75,-3.75) {$\underset{(0,1,2)}{\mathbb{Z}_{18}}$}                                        edge[<-] (a)   edge[<-] (b)    edge[<-] (c)   edge[<-] (e);
\node (i) at (0,-5) {$\underset{(0,0,0)}{\mathbb{Z}_{19}\rtimes\mathbb{Z}_{18}}$}     edge[<-] (f)    edge[<-] (g)     edge[<-] (h)    edge[<-] (d);

\node (1) at (6.75,0)	{\begin{tikzpicture}
\draw[thick,-] (0,0) -- (1,0);
\draw[thick, -] (0,0) -- (.5,-1);
\draw[thick, -] (1,0) -- (.5,-1);
\path[fill=gray!20, thick] (0,0) -- (.5,-1) -- (1,0) -- cycle;
\draw[thick, -] (0.2,-.45) -- (.8,-.45) node[pos = .5, above]{$1$};
\draw[thick, -] (0.2,-.45) -- (.8,-.45) node[pos = .5, below]{$2$};
\end{tikzpicture}};
\node (2) at (4.75,-1.5)	{\begin{tikzpicture}
\draw[thick,-] (0,0) -- (1,0);
\draw[thick, -] (0,0) -- (.5,-1);
\draw[thick, -] (1,0) -- (.5,-1);
\path[fill=gray!20, thick] (0,0) -- (.5,-1) -- (1,0) -- cycle;
\draw[thick, -] (0.2,-.45) -- (.8,-.45) node[pos = .5, above]{$3$};
\draw[thick, -] (0.2,-.45) -- (.8,-.45) node[pos = .5, below]{$0$};
\end{tikzpicture}};

\node (3) at (6.75,-1.5)	{\begin{tikzpicture}
\draw[thick,-] (0,0) -- (1,0);
\draw[thick, -] (0,0) -- (.5,-1);
\draw[thick, -] (1,0) -- (.5,-1);
\path[fill=gray!20, thick] (0,0) -- (.5,-1) -- (1,0) -- cycle;
\draw[thick, -] (0.2,-.45) -- (.8,-.45) node[pos = .5, above]{$4$};
\draw[thick, -] (0.2,-.45) -- (.8,-.45) node[pos = .5, below]{$3$};
\end{tikzpicture}};

\node (4) at (8.75,-1.5) {\begin{tikzpicture}
\draw[thick,-] (0,0) -- (1,0);
\draw[thick, -] (0,0) -- (.5,1);
\draw[thick, -] (1,0) -- (.5,1);
\path[fill=gray!20, thick] (0,0) -- (.5,1) -- (1,0) -- cycle;
\draw[thick, -] (0.2,.45) -- (.8,.45) node[pos = .5, above]{$2$};
\draw[thick, -] (0.2,.45) -- (.8,.45) node[pos = .5, below]{$0$};
\end{tikzpicture}};

\node (5) at (5.5,-3)	{\begin{tikzpicture}
\draw[thick,-] (0,0) -- (1,0);
\draw[thick, -] (0,0) -- (.5,-1);
\draw[thick, -] (1,0) -- (.5,-1);
\path[fill=gray!20, thick] (0,0) -- (.5,-1) -- (1,0) -- cycle;
\draw[thick, -] (0.2,-.45) -- (.8,-.45) node[pos = .5, above]{$2$};
\draw[thick, -] (0.2,-.45) -- (.8,-.45) node[pos = .5, below]{$1$};
\end{tikzpicture}};

\node (6) at (8,-3) {\begin{tikzpicture}
\draw[thick,-] (0,0) -- (1,0);
\draw[thick, -] (0,0) -- (.5,1);
\draw[thick, -] (1,0) -- (.5,1);
\path[fill=gray!20, thick] (0,0) -- (.5,1) -- (1,0) -- cycle;
\draw[thick, -] (0.2,.45) -- (.8,.45) node[pos = .5, above]{$2$};
\draw[thick, -] (0.2,.45) -- (.8,.45) node[pos = .5, below]{$1$};
\end{tikzpicture}};

\node (7) at (6.75,-4.5) {\begin{tikzpicture}
\draw[thick,-] (0,0) -- (1,0);
\draw[thick, -] (0,0) -- (.5,1);
\draw[thick, -] (1,0) -- (.5,1);
\path[fill=gray!20, thick] (0,0) -- (.5,1) -- (1,0) -- cycle;
\draw[thick, -] (0.2,.45) -- (.8,.45) node[pos = .5, above]{$0$};
\draw[thick, -] (0.2,.45) -- (.8,.45) node[pos = .5, below]{$0$};
\end{tikzpicture}};

\draw[->] (1) -- (2);
\draw[->] (1) -- (3);
\draw[->] (1) -- (4);
\draw[->] (1) -- (6);
\draw[->] (2) -- (5);
\draw[->] (2) -- (6);
\draw[->] (3) -- (5);
\draw[->] (3) -- (6);
\draw[->] (3) -- (7);
\draw[->] (4) -- (6);
\draw[->] (5) -- (7);
\draw[->] (6) -- (7);
\end{tikzpicture}
\caption{Structure digraph and simplify structure diagram for $GEN(F_{19})$.}
\end{figure}

\newpage
\section{Non-abelian group with all abelian subgroups}\label{sec4}
In this section, we study the achievement game and the avoidance game on finite non-abelian groups $G_\mathcal{A}$, in which all subgroups are abelian. We have some familiar examples of such groups as the symmetric group $S_3$ of order 6 and the Quaternion group $Q_8$ of order 8. Following are some results which will help us to understand the maximal subgroups of such groups.

\begin{theorem} \cite[{\S}{1}, p. 399]{5} 
The order of the group $G_\mathcal{A}$ has at most two distinct prime factors.  \end{theorem}

\begin{theorem} \cite[{\S}{1}, p. 399--402]{5}\label{abmax}
Let the order of the group $G_\mathcal{A}$ have two distinct prime factors $p,q$ such that $|G_\mathcal{A}|=p^\alpha q^\beta$, where $p< q$ and $\alpha,\beta$ are positive integers. Then the maximal subgroups are 
\begin{itemize}
\item[$(i)$] $q^\beta$ cyclic subgroups of order $p^\alpha$,  
\item[$(ii)$] a subgroup $P_{\alpha-1}\times Q_\beta$ of order $p^{\alpha-1}q^\beta$, where $P_{\alpha-1}$ is a cyclic subgroup of $G_\mathcal{A}$ of order $p^{\alpha-1}$ and $Q_\beta$ is the subgroup of $G_\mathcal{A}$ of order $q^\beta$.
\end{itemize} 
\end{theorem} 

\begin{corollary}  \cite[\S 1, p. 399--402]{5}\label{cor1}
Let the order of the group $G_\mathcal{A}$ be $p^\alpha q^\beta$, where $p<q$, $p$ and $q$ are distinct primes and $\alpha,\beta$ are positive integers. Then the Frattini subgroup of group $G_\mathcal{A}$ is a cyclic subgroup of order $p^{\alpha-1}$.
\end{corollary}

\begin{theorem} \cite[\S 1, p. 402--404 ]{5}\label{p+1} 
Let group $G_\mathcal{A}$ be such that $|G_\mathcal{A}|=p^\alpha$ for some prime $p$. Then $G_\mathcal{A}$ has $p+1$ maximal subgroups of order $p^{\alpha-1}$.  
\end{theorem}

\begin{proposition} Let group $G_\mathcal{A}$ be such that $|G_\mathcal{A}|=p^\alpha q^\beta$, where $p<q$, $p, q$ are two distinct primes, and $\alpha, \beta$ are positive integers. Then
\[  DGN(G_\mathcal{A}) =\left\{
\begin{array}{ll}
*0, &\text{if}\; p=2, q\neq2, \alpha\neq1 \\
*1, &\text{if}\; p\neq 2\neq q \\
*3, &	 otherwise
\end{array} 
\right. \]
\end{proposition}
\begin{proof}
Using the Theorem \ref{abmax}, the set $\mathcal{I}$ of intersection subgroups is given by
\[\mathcal{I}=\{P_{\alpha-1},P_\alpha,P_{\alpha-1}\times Q_\beta\}.\]

\textit{Case$(i)$.} Let $p\neq2\neq q$. Then the order of the group $G_\mathcal{A}$ is odd. Thus using \cite[Proposition 3.22, p. 11]{1}, we have $DGN(G_\mathcal{A})=*1$.

\textit{Case$(ii)$.} Let $p=2,q\neq2$ and $\alpha\neq1$. Then using the Corollary \ref{cor1}, we get the order of the Frattini subgroup $\Phi(G_\mathcal{A})$ is even. Thus by \cite[Proposition 3.23, p. 521]{1}, $DGN(G_\mathcal{A})=*0$.

\textit{Case$(iii)$.} Let $p=2$, $q\neq2$ and $\alpha=1$. Then $|G_{\mathcal{A}}| = 2q^{\beta}$ and $\mathcal{I}=\{\mathbb{Z}_1, P_{1} = \mathbb{Z}_{2},Q_\beta\}$. Now $P_1$ and $Q_\beta$ are maximal subgroups. Thus $type(X_{P_1})=(0,0,1)$ and $type(X_{Q_\beta})=(1,1,0)$. Note that $Opt(X_{\mathbb{Z}_1})=\{X_{P_1},X_{Q_\beta}\}$. This implies that $otype(X_{\mathbb{Z}_1})=\{(0,0,1),(1,1,0)\}$ and so, $type(X_{\mathbb{Z}_1})=(1,3,2)$. As $\mathbb{Z}_1$ is the Frattini subgroup, we have $DGN(G_\mathcal{A})=*3$.
\begin{figure}[h!]
\begin{center}
\begin{tikzpicture}
\node (a) at (7,0)	{\begin{tikzpicture}
\draw[thick,-] (0,0) -- (1,0);
\draw[thick, -] (0,0) -- (.5,-1);
\draw[thick, -] (1,0) -- (.5,-1);
\path[fill=gray!20, thick] (0,0) -- (.5,-1) -- (1,0) -- cycle;
\draw[thick, -] (0.2,-.45) -- (.8,-.45) node[pos = .5, above]{$3$};
\draw[thick, -] (0.2,-.45) -- (.8,-.45) node[pos = .5, below]{$2$};
\end{tikzpicture}};
\node (b) at (8.25,-1.5) {\begin{tikzpicture}
\draw[thick,-] (0,0) -- (1,0);
\draw[thick, -] (0,0) -- (.5,1);
\draw[thick, -] (1,0) -- (.5,1);
\path[fill=gray!20, thick] (0,0) -- (.5,1) -- (1,0) -- cycle;
\draw[thick, -] (0.2,.45) -- (.8,.45) node[pos = .5, above]{$1$};
\draw[thick, -] (0.2,.45) -- (.8,.45) node[pos = .5, below]{$0$};
\end{tikzpicture}};

\node (c) at (5.75,-1.5)	{\begin{tikzpicture}
\draw[thick,-] (0,0) -- (1,0);
\draw[thick, -] (0,0) -- (.5,-1);
\draw[thick, -] (1,0) -- (.5,-1);
\path[fill=gray!20, thick] (0,0) -- (.5,-1) -- (1,0) -- cycle;
\draw[thick, -] (0.2,-.45) -- (.8,-.45) node[pos = .5, above]{${1}$};
\draw[thick, -] (0.2,-.45) -- (.8,-.45) node[pos = .5, below]{$0$};
\end{tikzpicture}};
\draw[->] (a) -- (b);
\draw[->] (a) -- (c);
\end{tikzpicture}
\end{center}
\caption{Structure diagram of $DGN(G_{\mathcal{A}})$.}
\end{figure}

\end{proof}

\begin{proposition}\label{dngp}
Let $G_\mathcal{A}$ be a group such that $|G_\mathcal{A}| = p^\alpha$ for some prime $p$. Then
\[  DGN(G_\mathcal{A}) =\left\{
\begin{array}{ll}
*0, & p=2 \\
*1, & p\neq 2 
\end{array} 
\right. .\] 
\end{proposition}
\begin{proof} 
Let $p$ be a prime such that $|G_\mathcal{A}|=p^\alpha$. Then either $p=2$ or $p\neq2$. If $p\neq 2$, then the order of the group $G_{\mathcal{A}}$ is odd. Thus, using \cite[Propostion 3.22, p. 521]{1}, we get $DGN(G_\mathcal{A})=*1$.

Now, let $p=2$. Then using the Theorem \ref{p+1}, the group $G_\mathcal{A}$ has 3 maximal subgroup of order $2^{\alpha-1}$. Let these subgroups be $G_{1}, G_{2}$ and $G_{3}$. Now, we prove that the Frattini subgroup $\Phi(G_\mathcal{A})$ is non-trivial. If possible, suppose that $\Phi(G_\mathcal{A})$ is trivial. Then we have two cases either the subgrups $G_{1}\cap G_{2}\cap G_{3}$ is trivial or any two of $G_{1}, G_{2}$ and $G_{3}$ have non-trivial intersection but $\Phi(G_\mathcal{A})$ is trivial.

\textit{Case$(i)$.} Let $G_{1}\cap G_{2}\cap G_{3}$ be trivial. Then $G_{1}\cup G_{2}\cup G_{3} = G_{\mathcal{A}}$. Therefore, $2^{\alpha} = 3\cdot(2^{\alpha-1}-1)+1$ which implies that $\alpha = 2$. This is a contradiction to the fact that $G_{\mathcal{A}}$ is a non-abelian group.

\textit{Case$(ii)$.} Let $|G_{1}\cap G_{2}| = 2^{m}$ for some $m\le \alpha-2$ such that $\Phi(G_\mathcal{A})$ is trivial. Then $G_{1}\cup G_{2}\cup G_{3} = G_{\mathcal{A}}$. Therefore, $2^{\alpha} = (2\cdot2^{\alpha-1}-2^{m}-1)+1 + (2^{\alpha-1}-1)$ which implies that $\alpha -1 = m$. This is a contradiction.

Thus both the cases are not possible. Therefore, $\Phi(G_\mathcal{A})$ is non-trivial and so, the order of $\Phi(G_\mathcal{A})$ is even. Hence, using \cite[Propostion 3.23, p. 521]{1}, $DGN(G_\mathcal{A}) =*0$.
\end{proof}

\begin{proposition} Let group $G_\mathcal{A}$ be such that $|G_\mathcal{A}|=p^\alpha q^\beta$, where $p<q$, $p, q$ are distinct primes and $\alpha,\beta$ are positive integers. Then
\[  GEN(G_\mathcal{A}) =\left\{
\begin{array}{ll}
*0, &\; \text{if}\; p = 2, q\ne 2, \alpha \ne 1\\
*1\; \text{or}\;*2, & \; \text{if}\; p \ne 2\ne q\\
*3, & \; \text{otherwise}\quad .
\end{array} 
\right. \] 
\end{proposition}
\begin{proof}Let the order of the group $G_\mathcal{A}$ be $p^\alpha q^\beta$, where $p, q$ are distinct primes and $\alpha,\beta$ are positive integers.  Using the Theorem \ref{abmax}, the set $\mathcal{I}$ of intersection subgroups is given by
\[\mathcal{I}=\{P_{\alpha-1},P_\alpha,P_{\alpha-1}\times Q_\beta\}.\]

\textit{Case$(i)$.} Let $p\neq2\neq q$. Then $G_\mathcal{A}$ is a group of odd order. Thus using \cite[Corollary 4.8, p. 523]{1}, $GEN(G_\mathcal{A})$ is $*1$ or $*2$.

\textit{Case$(ii)$.} Let $p=2,q\neq2$ and $\alpha\neq 1$. Then $G_{\mathcal{A}}$ is an even order group and so, $type$ of terminal class is $type(X_{G_\mathcal{A}})=(0,0,0)$. The semi-terminal structure classes are the maximal subgroups namely $P_\alpha$ and $P_{\alpha-1}\times Q_{\beta}$ with $otype(X_{P_{\alpha}}) =\{(0,0,0)\} = otype(X_{P_{\alpha-1}\times Q_\beta})$. So, $type(X_{P_{\alpha}})=(0,1,2) = type(X_{P_{\alpha-1}\times Q_\beta})$. Note that the  non-terminal structure class corresponds to the Frattini Subgroup $P_{\alpha-1}$ of the group $G_{\mathcal{A}}$. Also, $Opt(X_{P_{\alpha-1}}) = \{X_{P_{\alpha}}, X_{P_{\alpha-1}\times Q_\beta}\}$. This implies that $otype(X_{P_{\alpha-1}}) = \{(0,1,2)\}$ and so, $type(X_{P_{\alpha-1}}) = (0,0,2)$. As $P_{\alpha-1}$ is the Frattini subgroup, $GEN(G_\mathcal{A}) = *0$.

\textit{Case$(iii)$.} Let $p=2, q\neq2$ and $\alpha=1$. Then the set $\mathcal{I}$ of intersection of maximal subgroups is $\mathcal{I} = \{\mathbb{Z}_1,P_1,Q_\beta\}$. Clearly, the $type$ of terminal class is $type(X_{G_\mathcal{A}}) = (0,0,0)$. The semi-terminal structure classes are the maximal subgroups namely $P_1$ and $Q_\beta$ with  $otype(X_{P_1}) = \{(0,0,0)\} = otype(X_{Q_\beta})$. Thus $type(X_{P_1})=(0,1,2)$ and $type(X_{Q_\beta})=(1,2,1)$. For the non-terminal structure class $X_{\mathbb{Z}_1}$ the option set is $\{X_{P_1},X_{Q_\beta}\}$. This implies $otype(X_{\mathbb{Z}_1}) = \{(0,1,2),(1,2,1)\}$. Hence $type(X_{\mathbb{Z}_1}) = (1,3,0)$. As $\mathbb{Z}_1$ is the Frattini subgroup, $GEN(G_\mathcal{A}) = *3$.
\begin{figure}[h!]
\begin{center}
\begin{tikzpicture}
\node (a) at (6,0)	{\begin{tikzpicture}
\draw[thick,-] (0,0) -- (1,0);
\draw[thick, -] (0,0) -- (.5,-1);
\draw[thick, -] (1,0) -- (.5,-1);
\path[fill=gray!20, thick] (0,0) -- (.5,-1) -- (1,0) -- cycle;
\draw[thick, -] (0.2,-.45) -- (.8,-.45) node[pos = .5, above]{$3$};
\draw[thick, -] (0.2,-.45) -- (.8,-.45) node[pos = .5, below]{$0$};
\end{tikzpicture}};
\node (b) at (7.25,-1.5) {\begin{tikzpicture}
\draw[thick,-] (0,0) -- (1,0);
\draw[thick, -] (0,0) -- (.5,1);
\draw[thick, -] (1,0) -- (.5,1);
\path[fill=gray!20, thick] (0,0) -- (.5,1) -- (1,0) -- cycle;
\draw[thick, -] (0.2,.45) -- (.8,.45) node[pos = .5, above]{$2$};
\draw[thick, -] (0.2,.45) -- (.8,.45) node[pos = .5, below]{$1$};
\end{tikzpicture}};

\node (c) at (4.75,-1.5)	{\begin{tikzpicture}
\draw[thick,-] (0,0) -- (1,0);
\draw[thick, -] (0,0) -- (.5,-1);
\draw[thick, -] (1,0) -- (.5,-1);
\path[fill=gray!20, thick] (0,0) -- (.5,-1) -- (1,0) -- cycle;
\draw[thick, -] (0.2,-.45) -- (.8,-.45) node[pos = .5, above]{${2}$};
\draw[thick, -] (0.2,-.45) -- (.8,-.45) node[pos = .5, below]{$1$};
\end{tikzpicture}};

\node (d) at (6,-3) {\begin{tikzpicture}
\draw[thick,-] (0,0) -- (1,0);
\draw[thick, -] (0,0) -- (.5,1);
\draw[thick, -] (1,0) -- (.5,1);
\path[fill=gray!20, thick] (0,0) -- (.5,1) -- (1,0) -- cycle;
\draw[thick, -] (0.2,.45) -- (.8,.45) node[pos = .5, above]{$0$};
\draw[thick, -] (0.2,.45) -- (.8,.45) node[pos = .5, below]{$0$};
\end{tikzpicture}};
\draw[->] (a) -- (b);
\draw[->] (a) -- (c);
\draw[->] (b) -- (d);
\draw[->] (c) -- (d);
\end{tikzpicture}
\end{center}
\caption{Structure diagram of $GEN(G_{\mathcal{A}})$.}
\end{figure}
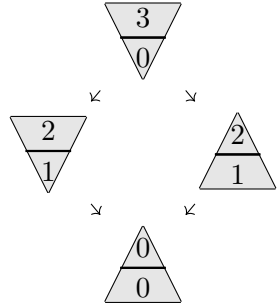
\end{proof}

\begin{proposition} Let $G_\mathcal{A}$ be a group such that $|G_\mathcal{A}| = p^\alpha$ for some prime $p$ and $\alpha > 2$ is a positive integer. Then
\[  GEN(G_\mathcal{A}) =\left\{
\begin{array}{ll}
*0, & p=2 \\
*1\;\textrm{or}\;*2, & p\neq 2 \quad.
\end{array} 
\right. \] 
\end{proposition}
\begin{proof}Let $|G_\mathcal{A}|=p^\alpha$, where $p$ is a prime and $\alpha > 2$ is a positive integer. If $p\neq 2$, then $G_{\mathcal{A}}$ is an odd order group. Therefore, using \cite[Corollary 4.8, p. 523]{1}, $GEN(G_\mathcal{A})$ is either $*1$ or $*2$.

Now, let $p=2$. Then, as observed in the Proposition \ref{dngp}, the Frattini subgroup $G_{\mathcal{A}}$ is of even order. Also, $X_{\Phi({G_\mathcal{A}})}$ is a non-terminal structure class. Thus using the Propostion \ref{p4.4}, $type\left(X_{\Phi(G_\mathcal{A})}\right) = (0,0,1)$ or $type\left(X_{\Phi(G_\mathcal{A})}\right) = (0,0,2)\}$. Hence $GEN(G_\mathcal{A})=*0$.
\end{proof}

\section*{Declarations}
The authors have no conflict of interest.

%

			\end{document}